\documentclass[a4paper]{article}
\usepackage{a4wide}
\usepackage{amsfonts}
\usepackage{amsmath}
\usepackage{mathtools}
\usepackage{amssymb}
\usepackage{array}
\usepackage{multicol}
\usepackage{wrapfig}
\usepackage{amsthm}
\usepackage{dsfont}
\usepackage[utf8]{inputenc}
\usepackage[T1]{fontenc}
\usepackage{caption}
\usepackage{subcaption}
\usepackage{enumerate} 
\usepackage{xifthen}
\usepackage[urlcolor=blue,colorlinks=false]{hyperref}
\usepackage{algorithm}
\usepackage{graphicx}
\usepackage{color}
\usepackage{pgf}

\usepackage{multirow}
\usepackage{tikz-cd} 
\usepackage{multicol}
\usepackage{tabularx}
\usepackage{diagbox}
\usepackage{longtable}
\usepackage{listings} 
\usepackage{cleveref}
\usepackage{tikz}
\usetikzlibrary{math}
\usetikzlibrary{patterns}
\usepackage{graphicx}
\usepackage{wrapfig}
\usepackage{changes} 
\usepackage[export]{adjustbox}

\newcommand{\N}{\ensuremath{\mathbb{N}}}

\renewcommand{\o}{\omega}

\newcommand{\Z}{\ensuremath{\mathbb{Z}}}
\newcommand{\Q}{\ensuremath{\mathbb{Q}}}
\newcommand{\R}{\ensuremath{\mathbb{R}}}
\newcommand{\C}{\ensuremath{\mathbb{C}}}
\newcommand{\K}{\ensuremath{\mathbb{K}}}

\newcommand{\F}{\mathbb{F}}

\newcommand{\Fix}{\operatorname{Fix}}
\newcommand{\ord}{\operatorname{ord}}

\usepackage{algpseudocode}
\usepackage{algorithm}

\newtheorem{theorem}{Theorem}[]
\newtheorem{lemma}[theorem]{Lemma}
\newtheorem{corollary}[theorem]{Corollary}
\newtheorem{proposition}[theorem]{Proposition}
\newtheorem{remark}[theorem]{Remark}
\newtheorem{definition}[theorem]{Definition}
\newtheorem{example}[theorem]{Example}

\newenvironment{Corollary}{\goodbreak \begin{corollary}\rmfamily}{\end{corollary}}

\numberwithin{equation}{section}
\numberwithin{table}{section}
\numberwithin{figure}{section}

\newcommand{\bend}{\hspace*{0ex} \hfill \hbox{\vrule height
    1.5ex\vbox{\hrule width 1.4ex \vskip 1.4ex\hrule  width 1.4ex}\vrule
    height 1.5ex}}

\long\def\symbolfootnote[#1]#2{\begingroup%
\def\thefootnote{\fnsymbol{footnote}}\footnote[#1]{#2}\endgroup}

\crefname{lemma}{Lemma}{Lemmata}
\crefname{definition}{Definition}{Definitions}
\crefname{theorem}{Theorem}{Theorems}
\crefname{corollary}{Corollary}{Corollaries}
\crefname{equation}{}{}
\crefname{remark}{Remark}{Remarks}
\crefname{algorithm}{Algorithm}{Algorithms}
\crefname{chapter}{Chapter}{Chapters}
\crefname{section}{Section}{Sections}
\crefname{table}{Table}{Tables}
\crefname{figure}{Figure}{Figures}
\crefname{example}{Example}{Examples}
\crefname{appendix}{Appendix}{Appendices}

\renewcommand{\thefootnote}{\fnsymbol{footnote}}

\allowdisplaybreaks

\title{Real and finite field versions of Chebotar\"ev's theorem}

\author{Tarek Emmrich \and Stefan Kunis}
\date{}

\begin{document}

\maketitle

\begin{abstract}
  Chebotarev's theorem on roots of unity states that all minors of the Fourier matrix of prime size are non-vanishing.
  This result has been rediscovered several times and proved via different techniques.
  We follow the proof of Evans and Isaacs and generalize the original result to a real version and a version over finite fields.
  For the latter, we are able to decrease a sufficient lower bound on the characteristic by Zhang considerably.
  Direct applications include a specific real phase retrieval problem as well as a recent result for Riesz bases of exponentials.
  
   
 	\noindent\emph{2020 AMS Mathematics Subject Classification}:
    \text{%
    11T06, 
    15B33,  
    15B99, 
    43A25, 
    65T50. 
}
\end{abstract}

\section{Introduction}
Let $p \in \N$ be a prime and $\o_p \in \K$ be a primitive $p$-th root of unity for an arbitrary field $\K$, then we consider
the Fourier matrix
\[
F_p=
\left(\omega_p^{jm}\right)_{{j,m=0,\ldots,p-1}}=
\begin{pmatrix}
        1 & 1 & 1 & \ldots & 1\\
        1 & \omega_p & \omega_p^2 & \ldots & \omega_p^{p-1}\\
        1 & \omega_p^2 & \omega_p^4 & \ldots & \omega_p^{p-2} \\
        1 & \vdots& \vdots & \ddots & \vdots \\
        1 & \omega_p^{p-1} & \omega_p^{p-2} & \ldots & \omega_p
    \end{pmatrix} \in \K^{p \times p},
\]
and ask under which conditions all of its square submatrices are non-singular.
Over the complex number field, the well known theorem by Chebotar\"ev gives an affirmative answer.
\begin{theorem}[See e.g. \cite{Evans,Stevenhagen,frenkel,Tao}; see \cref{sec:pre}]\label{Chebotarev}
   Let $p\in\N$ be prime and $\o_p=\exp(2\pi i / p)\in\C$ be a $p$-th root of unity, then
   all minors of the Fourier matrix \[F_p\in\C^{p\times p}\] are nonzero.
\end{theorem}
We will generalize this theorem to a real version and a version over finite fields as follows:
\begin{theorem}[See \cref{sect:real}]\label{th:introReal}
    Let $p\in\N$ be prime and $\o_p=\exp(2\pi i / p)\in\C$ be a $p$-th root of unity, then
    all minors of the real Vandermonde matrix 
    \begin{align*}
        \left(\left(\omega_p^j+\omega_p^{-j}\right)^m\right)_{{j,m=1,\ldots,\frac{p-1}{2}}}
     \in\R^{\frac{p-1}{2}\times\frac{p-1}{2}}   
    \end{align*}
    are nonzero. 
\end{theorem}

\begin{theorem}[Informal version; see \cref{sec:finitefields}]\label{th:introFiniteF}
    Let $p,q\in\N$ be two distinct primes, $\ord_p(q)=p-1$, $\o_p\in\F_{q^{p-1}}$ be a $p$-th root of unity, and $q$ sufficiently large, then all minors of the Fourier matrix \[F_p\in\F_{q^{p-1}}^{p\times p}\]
    are nonzero. 
\end{theorem}

The matrix in \cref{th:introReal} is a real Vandermonde matrix and the result follows by understanding the Galois group of the field extension $\Q \subseteq \Q(\o_p+\o_{p}^{-1}) \subseteq \Q(\o_p)$.
Further note that Vandermonde matrices are known to be totally positive for positive nodes, see e.g.~\cite{Pi10} or \cref{Schurevaluation}, and there is also a variant for a discrete cosine transform matrix which easily follows directly from \cref{Chebotarev}. 

Regarding \cref{th:introFiniteF}, we note the following.
\begin{enumerate}
    \item In \cite{FiniteFieldChebFalse} it has been falsely stated that the criterion $\ord_p(q)=p-1$ is a sufficient condition (without any lower bound on $q$);
    Zhang \cite[Example 4.1,4.2]{Zhang} gave counterexamples for $p=11, q=2$ and $p=41, q=11$.
    \item Zhang \cite[Theorem A]{Zhang} proved \cref{th:introFiniteF} under the conditions $\ord_p(q)=p-1$ and $q > \Gamma$, for some large $\Gamma$.
    \item Recently, Pfander, Caragea, Lee, and Malikiosis \cite{PfCaLeMa25} improved Zhang's result by getting rid of the criterion $\ord_p(q)=p-1$.
    \item Here, we decrease the bound $\Gamma$ on the characteristic $q$ considerably.
\end{enumerate}




Beyond scientific curiosity, our results are motivated by the following applications: The real versions in \cref{VandermondeCosine} and \cref{rem:Cos} are of interest if one wishes to do computations avoiding complex arithmetic. In particular, the complex Fourier matrix is \emph{an} eigenvector matrix of the real symmetric graph Laplacian of the ring graph, whereas the real eigenvectors consisting of sine/cosine-pairs are the more natural choice. For the analysis of sparse signals  on (more general) graphs, we refer to \cite{Em24,EmJuKu25}. A second application is in real phase retrieval for sparse polynomials in monomial or Chebyshev basis.
The finite field version yields an uncertainty principle for specific  integer matrices via standard arguments, see e.g.~\cite[Proof of Thm.~1.1]{Tao} and \cref{ex:F5}. More surprisingly, a recent question in time-frequency analysis on Riesz bases of exponentials studies the principal minors of the complex Fourier matrix of size $pq$ and reduces this to the nonvanishing minors property of the Fourier matrix of size $p$ in characteristic $q$, see \cite{carageaarticle,loukaki}. Consequently, \cref{FourierfiniteGreater} also weakens the conditions of these results.

This paper is organized as follows: In \cref{sec:pre}, we recall the proof Chebotar\"ev's theorem given by Evans and Isaacs \cite{Evans} and slightly expand it in a way, such that it can be generalized directly to a real version in \cref{sect:real}. For this proof Schur polynomials are essential.
Now, the known proofs always use a divisibility argument, that is valid in $\Z$ but not in characteristic $q>0$. Thus a more delicate analysis is required in \cref{sec:finitefields}. From this analysis we also obtain a new proof of \cref{Chebotarev} that completely works in the ring $\Z(\o_p)$ and does not rely on a substitution $\o_p \mapsto 1$, which is a standard trick in the known proofs, see \cref{remarkProof}.

\section{Prerequisites}\label{sec:pre}
An important tool for our proofs are Schur polynomials. For $k\in\N$, a set $A=\{a_1,\ldots,a_k\} \in \binom{\{0,1,\ldots,n-1\}}{k}$, with $a_1<a_2<\ldots<a_k$, and any field $\K$ consider the generalized Vandermonde matrix
\begin{align}\label{eq:V}
V_A=V_A(x_1,\ldots,x_k)=
    \begin{pmatrix}
        x_1^{a_1} & x_2^{a_1} & x_3^{a_1} & \ldots & x_k^{a_1}\\
        x_1^{a_2} & x_2^{a_2} & x_3^{a_2} & \ldots & x_k^{a_2}\\
        \vdots & \vdots & \vdots & \ddots & \vdots \\
        x_1^{a_k} & x_2^{a_k} & x_3^{a_k} & \ldots & x_k^{a_k}
    \end{pmatrix}
    \in (\K[x_1,\ldots,x_k])^{k \times k},
\end{align}
and drop the subscript for $A=\{0,\hdots,k-1\}$.
For any matrix $M$, denote by $M_{A,B}$ the restriction of $M$ to the rows $A$ and the columns $B$. With this notation we observe $V_A(x_1,\ldots,x_k)=(V_{\{0,\ldots,n-1\}}(x_1,\ldots,x_n))_{A,\{1,\ldots,k\}}$.

The determinant $\det(V_A)$ vanishes if $(x_i=x_j)$ and thus $(x_j-x_i)| \det(V_A)$. This motivates the definition of the Schur polynomial $s_A(x_1,\ldots,x_k)$ as
\[
s_A(x_1,\ldots,x_k)=\frac{\det(V_A)}{\prod\limits_{1 \leq i < j \leq k}(x_j-x_i)}.
\]
The following results is due to Mitchell \cite{Mitchell}, but is stated in a modern way in \cite{Evans}.
\begin{lemma}[See \cite{Mitchell,Evans}]\label{Schurevaluation}
    Let $A \in \binom{\{0,1,\ldots,n-1\}}{k}$ be any index set. The Schur polynomial $s_A(x_1,\ldots,x_k) \in \Q[x_1,\ldots,x_k]$ has nonnegative coefficients and
    \[
    s_A(1,\ldots,1)=\frac{\prod\limits_{1\leq i < j \leq k}(a_j-a_i)}{\prod\limits_{1\leq i < j \leq k}(j-i)}.
    \]
\end{lemma}

These tools are already sufficient to prove \cref{Chebotarev}. We will provide a proof that might seem cumbersome at some points, but can directly be generalized to a real version.

\begin{proof}[Proof of \cref{Chebotarev}]
    Let $A,B \in \binom{\{0,1,\ldots,p-1\}}{k}$ be the indices of a submatrix
    \[
    M=
    \begin{pmatrix}
        \omega_p^{a_1b_1}&\o_p^{a_1b_2}& \ldots & \o_p^{a_1b_k}\\
        \o_p^{a_2b_1}&\o_p^{a_2b_2}& \ldots & \o_p^{a_2b_k}\\
        \vdots & \vdots & \ddots & \vdots\\
        \o_p^{a_kb_1}&\o_p^{a_kb_2}&\ldots&\o_p^{a_kb_k}
    \end{pmatrix}
    \]
    of the Fourier matrix $F_p$, with $\o_p=\exp(2 \pi i /p)$. We consider the Schur polynomial $s_A(x_1,\ldots,x_k)$ and observe that $\det(M)=0$ if and only if $s_A(\o_p^{b_1},\o_p^{b_2},\ldots,\o_p^{b_k})=0$. After the substitution 
    \[\alpha_B \colon \Q[x_1,\ldots,x_k] \to \Q[X] \text{, }x_l \mapsto X^{b_l},\]
    the multivariate Schur polynomial $s_A$ becomes an univariate polynomial $\alpha_B(s_A) \in \Q[X]$. The prior statement can be rewritten as $\alpha_B(s_A)(\o_p)=0$ if and only if $\det(M)=0$.
    Now, since $\Q[X]$ is a principal ideal domain, $\alpha_B(s_A)(\o_p)=0$ implies that $\alpha_B(s_A)$ is divisible by the minimal polynomial $\Phi_p(X)=\sum_{i=0}^{p-1}X^i$ of $\o_p$, i.e.,
    \begin{align}\label{factorcheb}
    \alpha_B(s_A)=\Phi_p(X) \cdot R(X),
    \end{align}
    for some $R \in \Q[X]$. Both polynomials, $\alpha_B(s_A)$ and $\Phi_p$, have integer coefficients and leading coefficient 1. Hence the same holds for $R$. Equality \cref{factorcheb} also holds true after another substitution $X \mapsto 1$ and yields
    \begin{align}\label{factorprime}
        \alpha_B(s_A)(1)=\underbrace{\Phi_p(1)}_{=p} \cdot R(1).
    \end{align}
Finally, we want to show that $\alpha_B(s_A)(1)$ is not divisible by the prime $p$, which contradicts $\det(M)=0$. Here we can use that the value $X=1$ has good properties under the substitution $x_l \mapsto X^{b_l}$, i.e.,
    \begin{align*}
    \alpha_B(s_A)(1)
    =s_A(1,1,\ldots,1)=\frac{\prod\limits_{1\leq i < j \leq k}(a_j-a_i)}{\prod\limits_{1\leq i < j \leq k}(j-i)},
    \end{align*}
    by \cref{Schurevaluation}. The factors $(a_j-a_i)$ are all smaller than $p$ and cannot divide $p$. This contradicts equation \cref{factorprime} and hence $\det(M) \ne 0$.
\end{proof}

A key ingredient in the proof is that there exists a  substitution $\alpha_B \colon \Q[x_1,\ldots,x_k] \to \Q[X]$ with $\alpha(s_A) \in \Q[X]$ and $s_A(\o_p^{b_1},\ldots,\o_p^{b_k})=\alpha_B(s_A)(\o_p)$. This substitution only exists because of the polynomial connection between the roots $\o_p^{b_i}$ that is encoded in the fact that $\Q(\o_p)=\Q[X]/\Phi_p(X)$ is the splitting field of $\Phi_p(X)$.
In this case we have $\operatorname{Gal}(\Q[X]/\Phi_p(X)|\Q)=(\Z/p)^\times$.
\begin{figure}[ht]
    \centering
        \[\begin{tikzcd}
            s_A(x_1,\ldots,x_k) \arrow{r}{x_l\mapsto X^{b_l}}\arrow[swap]{dr}{x_l \mapsto 1}  &    \alpha_B(s_A) \arrow{d}{X \mapsto 1} \\
             & \frac{\prod\limits_{1\leq i < j \leq k}(a_j-a_i)}{\prod\limits_{1\leq i < j \leq k}(j-i)}
            \end{tikzcd}
        \]
    \caption{The Substitutions of the Schur Polynomial in the proof of \cref{Chebotarev}.}
    \label{CommutingSchurDiagramm}
\end{figure}
Furthermore, we were able to evaluate at the point $x_1=\ldots=x_k=X=1$, here it was crucial that $1$ is a fixed point of the maps $x_l \mapsto X^{b_l}$ and the diagram in  \cref{CommutingSchurDiagramm} commutes.

\section{Real version of Chebotar\"ev's theorem}\label{sect:real}

For each prime $p>3$ the group $(\Z/p)^\times$ is the Galois group of the field extension $\Q \subseteq \Q[X]/\Phi_p(X)$ and has the subgroup $H=\{1,-1\}$. This subgroup $H$ is generated by the automorphism on $\Q[X]/\Phi_p(X)$ which is induced by the complex conjugation $i \mapsto -i$. By the Galois correspondence this induces the subfield
\[
\Q \subsetneq \Fix(H) \subsetneq \Q[X]/\Phi_p(X),
\]
of degree $\deg_{\Q}(\Fix(H))=\frac{p-1}{2}$. This field is generated as $\Q(\o_p+\o_p^{-1})=\Fix(H) \subseteq \R$ and has the cyclic Galois group
\[\Z_p^\times/H\cong \Z_{\frac{p-1}{2}},
\]
which is generated by the image of any generator $(\Z/p)^\times$ modulo $H$. It is $\o_p+\o_p^{-1}=2 \cos (2 \pi /p)$ and these elements are mapped to $(\o_p^j+\o_p^{-j})=2\cos(2 \pi j /p)$ by the Galois group of the cyclotomic field. Thus the elements
\[
(\o_p^{j}+\o_p^{-j})=2 \cos (2 \pi j/p), \quad j=1,\ldots,\frac{p-1}{2},
\]
are conjugate and share the same minimal polynomial
\[
P(X)=\prod_{j=1}^{(p-1)/2}(X-2\cos(2 j \pi /p)).
\]
The coefficients of this polynomial are symmetric in the cosine terms and hence, due to the Galois correspondence, must lie in the fixed field of the Galois group, i.e., the coefficients are rational. An explicit formula for the coefficients of this minimal polynomial is given in \cite{CosineMinimalPolynomial} by
    \[
    P(x)=\sum\limits_{0\leq k \leq \frac{n}{2}}(-1)^k\binom{n-k}{k}X^{n-2k}+\sum\limits_{0 \leq k < \frac{n}{2}}(-1)^k \binom{n-k-1}{k}X^{n-2k-1},
    \]
where $n=\frac{p-1}{2}$. However, this expression does not provide much additional insight into the structure of the minimal polynomial and will not be used later.
In order to use the Schur polynomials, we need to understand the polynomial expression between the roots.
\begin{lemma}\label{Galoisgroupcosine}
    Let $p$ be a prime and let $P(X) \in \Q[X]$ be the minimal polynomial of $2\cos(2 \pi /p)$. The elements $\varphi_j \in \Q[X]/P(X)$ that are given by the recursion
    \begin{align*}
    \varphi_1(X)&=X,\\
    \varphi_2(X)&=X^2-2,\\
    \varphi_{j}(X)&=X \cdot \varphi_{j-1}(X)-\varphi_{j-2}(X),\qquad
    j=3,\ldots,\frac{p-1}{2}, 
    \end{align*}
    are to the roots of $P$. Viewed over the rationals, the polynomials $\varphi_j(X)$ have the common fixed point at $2$, i.e., $\varphi_j(2)=2$ for all $j$.
\end{lemma}

\begin{proof}
    In the field $\Q[X]/P(X)$ we can use $X=(\o_p+\o_p^{-1})$ and we have
    \[X^2-2=\o_p^2+\o_p^{-2}.\]
    Induction on $j$ proves the first part, where we use
    \begin{align*}
        \varphi_j(X) &= (\o_p+\o_p^{-1}) \cdot (\o_p^{j-1}+\o_p^{1-j}) - (\o_p^{j-2}+\o_p^{2-j})
        = \o_p^j + \o_p^{-j}.
    \end{align*}
    Regarding the second part, we see that $\varphi_j(2)=2 \cdot \varphi_{j-1}(2)-\varphi_{j-2}(2) 
    =2 \cdot 2-2
    =2$.
\end{proof}

Now, we evaluate the minimal polynomial $P(X)$ at the fixed point $2$ of the polynomials $\varphi_j$, $j=1,\ldots,\frac{p-1}{2}$.

\begin{lemma}\label{evaluationCosineMinPoly}
    Let $p$ be a prime. The minimal polynomial $P(X)$ of $2 \cos ( 2 \pi /p)$ has the evaluation $P(2)=p$.
\end{lemma}

\begin{proof}
 The Chebyshev polynomials of first kind $T_p\colon[-1,1]\to\R$ are given by
 \begin{align*}
     T_p( x)=\cos(p\arccos x).
 \end{align*}
 The roots of
 \begin{align*}
  P(X)=\prod_{j=1}^{(p-1)/2}\left(X-2\cos\frac{2 j \pi}{p}\right)
 \end{align*}
 coincide with the local maxima of $2T_p(X/2)$ in $(-2,2)$, $2T_p(X/2)=2$, and $2T_p(X/2)$ is a monic polynomial of degree $p$. Hence we have the identity
 \begin{align*}
  2(T_p(X/2)-1)=(P(X))^2(X-2)    
 \end{align*}
 from which also
 \begin{align*}
 (P(2))^2
 =\lim_{y\to 1} \frac{T_p(y)-1}{y-1}=T_p'(1)=p\lim_{\theta \to 0}\frac{\sin p\theta}{\sin\theta} = p^2    
 \end{align*}
 follows.
\end{proof}

We now have all the tools to state and prove the main theorem of this section, which can be seen as a real version of Chebatar\"ev's theorem. 

\begin{theorem}\label{VandermondeCosine}
    Let $p$ be a prime. All minors of the Vandermonde matrix 
    \[
    \left(\left(\cos\left(\frac{2 \cdot j \cdot \pi}{p}\right)\right)^{m}\right)_{j,m=1,\ldots,\frac{p-1}{2}}
    \in\R^{\frac{p-1}{2}\times\frac{p-1}{2}}
    \]
    are nonzero.
\end{theorem}

\begin{proof}
    For notational convenience, we work with the Vandermonde matrix, see Equation \eqref{eq:V},
    \[
    V\left(2 \cdot \cos \left(\frac{2 \cdot j \cdot \pi}{p}\right) _{j=1,\ldots,\frac{p-1}{2}}\right),
    \]
    whose columns are a multiple of the original matrix and hence the minors show the same vanishing behaviour. Let $A,B \in \binom{[\frac{p-1}{2}]}{l}$ be the indices of the submatrix that is under consideration. This minor is zero if and only if $s_A(2\cos(2 b_1 \pi /p),\ldots,2\cos(2 b_l \pi /p))=0$. We break down the step from $s_A(x_1,\ldots,x_k)$ to $s_A(2\cos(2 b_1 \pi /p),\ldots,2\cos(2 b_l \pi /p))=0$ in two steps. The first step is the substitution
    \[\alpha \colon \Q[x_1,\ldots,x_l] \to \Q[X] \text{, } x_i \mapsto \varphi_{b_i}(X),\]
    with $\varphi_{b_i}$ from \cref{Galoisgroupcosine}. The second step is to look at the image of $\alpha_B(s_A)$ modulo the minimal polynomial $P$ of $2\cos(2 \pi /p)$.
    The minor is zero if and only if $P|\alpha_B(s_A)$ in $\Q[X]$. If this equation holds, it will also hold after another substitution $X \mapsto 2$. \cref{evaluationCosineMinPoly} yields $P(2)=p$ and \cref{Schurevaluation} yields
    \begin{align}\label{EqualityCosine1}
    \alpha_B(s_A)(2,\ldots,2) &= s_A(2,\ldots,2)\\
    \label{EqualityCosine2}&=2^d \cdot  s_A(1,\ldots,1) \\
    \label{EqualityCosine3}&= 2^d \cdot \frac{\prod\limits_{1\leq i < j \leq l}(a_j-a_i)}{\prod\limits_{1\leq i < j \leq l}(j-i)},
    \end{align}
    which is a contradiction to the divisibility by $p$, hence the minor must be nonzero. Equality \cref{EqualityCosine1} follows from the fact that $2$ is a fixed point of all $\varphi_k$, equality \cref{EqualityCosine2} follows from the fact that Schur polynomials are homogeneous of some degree $d$ and equality \cref{EqualityCosine3} is \cref{Schurevaluation}.
\end{proof}

\begin{example}
 Let $p=7$ and $t_1=\cos(2\pi/7)\approx 0.62$, $t_2=\cos(4\pi/7)\approx -0.22$, and $t_3=\cos(6\pi/7)\approx -0.90$ be given. Then all minors of the Vandermonde matrix
 \begin{align*}
  \begin{pmatrix}
    1 & t_1 & t_2^2\\
    1 & t_2 & t_2^2\\
    1 & t_3 & t_3^2
  \end{pmatrix}
 \end{align*}
 are nonzero, while (obviously) the matrix is not totally positive.
\end{example}

\begin{remark}\label{rem:Cos}
 Another real version can be infered from the classical Chebotar\"ev result.
    Let $p$ be a prime, set $r=(p-1)/2$, and
\begin{align*}
    F&=\left(\exp\frac{2\pi \mathrm{i} kj}{p}\right)_{k,j=-r,\hdots,r}\in\C^{p\times p},
    \qquad
    C
    =\left(2\cos\frac{2\pi kj}{p}\right)_{k,j=1,\hdots,r}\in\C^{r\times r}.
\end{align*}
Now assume $\det(C_{J\times K})=0$ for some $J,K\in\binom{[r]}{s}$.
Then there exists a $\hat f=(\hat f_{-r},\hdots,\hat f_r)^{\top} \in\C^p$, $\hat f_k=\hat f_{-k}$, which fulfils 
\begin{align*}
 0=\sum_{k\in K} 2\hat f_k \cos\frac{2\pi kj}{p} = \sum_{k\in K\cup -K} \hat f_k \exp\frac{2\pi\mathrm{i} kj}{p}
\end{align*}
for $j\in J$, as well as for $j\in -J$ in the first equation since the cosine is even.
This implies $\det(F_{(J\cup -J)\times (K\cup -K)})=0$ which is false.

More general results on uncertainty principles for vectors with symmetry and related Chebotar\"ev theorems can be found in \cite{GaKaKa21}.
\end{remark}

Finally, we note in passing that polynomials that are either sparse in the monomial basis (using \cref{VandermondeCosine}) or the Chebyshev basis (using \cref{rem:Cos}) are uniquely determined by a minimal number of \emph{phaseless} evaluations. This is a direct consequence since the non-vanishing minors property implies the so-called complement property, which is equivalent to the uniqueness of the real phase-retrieval problem, see e.g.~\cite{BaCaEd06,BeEd22}.

\section{Chebotar\"ev's theorem over finite fields}\label{sec:finitefields}
Now let us once again consider the Fourier matrix $F_p$ of prime size, but this time over fields of characteristic $q$ for $q \in \N$ prime.
It is well known that, given two primes $p,q \in \Z$, the polynomial
\[
\overline{\Phi_p(X)}=\sum_{i=0}^{p-1}X^i \in \F_q[X]
\]
is not even necessarily irreducible. In fact $\overline{\Phi_p(X)}$ is irreducible if and only if $\ord_p(q)=p-1$.
This follows from the fact that, if $\Phi_p$ is irreducible, the roots of $\overline{\Phi_p(X)}$ are given by $X^{q^i} \in \F_q[X]/\Phi_(X)$, for $i=0,1,\ldots,p-2$, and the roots fulfill $X^p=1$. Thus the roots are given by
\[
X,X^{(q \mod p)},X^{(q^2 \mod p)},X^{(q^3 \mod p)},\ldots,X^{(q^{p-2} \mod p)},
\]
and these are repeated, if $\ord_p(q)=r< p-1$, which would contradict the separability of $\overline{\Phi_p(X)}$. In fact it is known that $\overline{\Phi_p(X)} \in \F_q[X]$ splits into $\frac{p-1}{r}$ many irreducible polynomials of degree $r$. The roots of all these polynomials are contained in the field $\F_{q^r}$.

\begin{example}
  For $p=7$ and $q=11$ we have $r=\ord_7(11)=3$ and the factorization
\[
(X^3+7X^2+6X+10)(X^3+5X^2+4X+10)=(X^6+X^5+X^4+X^3+X^2+X+1) \in \F_{11}[X].
\]
\end{example}
Zhang's conditions \cite{Zhang}, $\ord_p(q)=p-1$ and  $q > \Gamma$, for some large $\Gamma$, imply the nonvanishing of the minors, which on polynomial side again can be written as
\[
\overline{\Phi_p(X)} \nmid \alpha_B(s_A(x_1,\ldots,x_k)) \in \F_q[X],
\]
for all $A,B \in \binom{\{0,1,\ldots,p-1\}}{k}$. The sufficient lower bound is given by
\[
\Gamma = \max_{A \subseteq \{0,1,\ldots,p-1\}} \frac{\prod_{i < j}(a_j-a_i)}{\prod_{i < j}(j-i)},
\]
and for $q > \Gamma$, the Euclidean division
\[
\alpha_B(s_A(x_1,\ldots,x_k)) = f(X) \cdot\Phi_p(X) + R(X) \in \Z[X]
\]
in combination with
\[
\alpha_B(s_A(x_1,\ldots,x_k))(1)=\frac{\prod_{i < j}(a_j-a_i)}{\prod_{i < j}(j-i)},
\]
yields that $q$ is too large to divide all the coefficients of $R(X)$ and thus $\overline{R(X)} \neq 0 \in \F_q[X]$. We want to improve the condition on $q$ with an approach that provides more insight into the problem. From now on, let $p,q$ be two distinct primes with $\ord_p(q)=r$ and denote $\Phi_p(X)=\sum_{i=0}^{p-1}X^i \in \Z[X]$ and $\overline{\Phi_p(X)}=\sum_{i=0}^{p-1}X^i \in \F_q[X]$. Since $\overline{\Phi_p(X)}$ is not necessarily irreducible, we denote by $\overline{P(X)} \in \F_q[X]$ an irreducible factor of $\overline{\Phi_p(X)}$ and by $P(X) \in \Z[X]$ the canonical preimage of $\overline{P(X)}$ in $\Z[X]$. Note that all the nontrivial $(p-1)$ many $p$-th roots of unity $\o_p$ will be contained in the field $\F_{q^r} \cong \F_q[X]/\overline{P(X)}$, with one explicit root given by $\omega_p=X$. Let
\begin{align}\label{def:FourierFiniteField}
F_p=V(1,\o_p,\o_p^2,\o_p^3,\ldots,\o_p^{p-1}) \in \left(\F_{q^r}\right)^{p \times p}.
\end{align}
be the Fourier matrix of size $p$ in characteristic $q$. This matrix always has entries in $\F_{q^r}$. 
\begin{example}
Again, for $p=7$ and $q=11$ we know that $\F_{11^3}\cong \F_{11}[X]/(X^3+7X^2+6X+10)$ is the splitting field of $\overline{\Phi_p(X)}$ and one explicit $7$-th root of unity is $\o_7=X$. The Fourier matrix is given by
\[
F_7=V(1,X,X^2,X^3,X^4,X^5,X^6) \in (\F_{11}[X]/(X^3+7X^2+6X+10))^{7 \times 7},
\]
with $X^3=4X^2+5X+1,$ $X^4=-X^2-X-4,$ $X^5=-5X^2-X-1$ and $X^6=X^2-4X-5$ in the field $\F_{11^3}=\F_{11}[X]/(X^3+7X^2-6X+10)$.    
\end{example}
Furthermore, for $B \in \binom{\{0,1,\ldots,p-1\}}{k}$ and $l \in \{1,\ldots,p-1\}$ we denote \[l \cdot B := \{ (l \cdot b) \mod p \colon b \in B\}.\]
The multiplication with $l$ is a bijection on $\F_p^\times$.
We start with a simple well known observation. 

\begin{lemma}\label{shiftIndicesCharq}
    Let $p,q \in \Z$ be two distinct primes with $\ord_p(q)=r$. 
    Let $F_p\in \F_{q^r}^{p \times p}$ be the Fourier matrix in characteristic $q$, see \cref{def:FourierFiniteField}.
    Let $A,B \in \binom{\{0,1,\ldots,p-1\}}{k}$ be two index sets. We have
    \[
    \det((F_p)_{A,B})=0
    \]
    if and only if
    \[
    \det((F_p)_{A,q^l \cdot B})=0
    \]
    for all $l=0,\ldots,p-2$.
\end{lemma}

\begin{proof}
    The determinant $\det((F_p)_{A,B})$ is a polynomial expression in $\o_p$. The Galois group of the field extension $\F_p \subseteq \F_{q^r}$ is generated by $Y \mapsto Y^q$ and hence $\o_p,\o_p^q,\ldots,\o_p^{q^{p-2}}$ are conjugate. Thus the vanishing of the polynomials is equivalent.
\end{proof}

For our main results, from now on we consider two distinct primes $p,q$ with $\ord_p(q)=p-1$. Note the map
\[l \mapsto (q^l \mod p)\]
induces an isomorphism from $\{0,\ldots,p-2\}$ to $\{1,\ldots,p-1\}$.
\begin{figure}[h]
    \centering
        \[\begin{tikzcd}[row sep=large, column sep=large]
            \Z[X] \arrow{r}{ \mod \Phi_p(X)} \arrow{d}{\mod q}  &    \Z[X]/\Phi_p(X) \arrow{d}{\mod q} \\
             \F_q[X] \arrow{r}{\mod \overline{\Phi_p(X)}}  & \F_q[X]/\overline{\Phi_p(X)}
            \end{tikzcd}
        \]
    \caption{The $p$-th cyclotomic polynomial over $\Z$ and $\F_q$, with $\ord_p(q)=p-1$.}
    \label{fig:CyclotomicFinite}
\end{figure}
\newline
Consider the commuting diagram in \cref{fig:CyclotomicFinite}, we want to know whether 
\begin{align}\label{eq:finiteP|sA}
\overline{\Phi_p(X)} \mid \overline{\alpha_B(s_A(x_1,\ldots,x_k))} \in \F_q[X]/\overline{\Phi_p(X)}
\end{align}
for some $A,B \in \binom{\{0,1,\ldots,p-1\}}{k}$. We know by \cref{Chebotarev} that
\[
\Phi_p(X) \nmid \alpha_B(s_A(x_1,\ldots,x_k)) \in \Z[X]
\]
for all $A,B \in \binom{\{0,1,\ldots,p-1\}}{k}$. Euclidean division in $\Z[X]$ yields
\[
\alpha_B(s_A(x_1,\ldots,x_k)) = f(X) \cdot \Phi_p(X) + R(X),
\]
with $R\neq 0$ and thus equation \cref{eq:finiteP|sA} holds if $q$ divides all the coefficients of $R(X)$. The polynomial $R$ itself might be quite complicated, but averaging over well chosen column sets $B$ can be reasonably described. For two fixed sets $A,B \in \binom{\{0,1,\ldots,p-1\}}{k}$ denote by $m_{A,B}$ the number of indices $\mu$ of $s_A(x_1,\ldots,x_k)=\sum_{\mu \in T} c_\mu x^\mu$, such that $p$ divides $\sum_{j=1}^k\mu_jb_j$, i.e.,
\[
m_{A,B}=\sum_{\mu \in T} c_\mu \mathds{1}\left(p \big| \sum_{j=0}^k\mu_jb_j\right).
\]
Here $b_1<\ldots < b_k$ is the ascending list of the elements of $B$. By abuse of notation, we also write $\langle \mu , B \rangle=\sum_{j=1}^k\mu_jb_j$.

\begin{lemma}\label{EvaluationSumOfSchur}
    Let $p\in \Z$ be a prime and let $A,B \in \binom{\{0,1,\ldots,p-1\}}{k}$ be two sets.
    We have
    \[
    \sum_{l=1}^{p-1}\alpha_{(l \cdot B)}(s_A(x_1,\ldots,x_k)) = m_{A,B} \cdot p - \frac{\prod_{i<j}(a_j-a_i)}{\prod_{i<j}(j-i)} \in \Z[X]/\Phi_p(X).
    \]
\end{lemma}

\begin{proof}
    The ring $\Z[X]/\Phi_p(X)$ is a $\Z$-algebra and all substitutions $\alpha_{(q^l\cdot B)}$ are $\Z$-linear maps. Thus we observe
    \begin{align*}
        \sum_{l=1}^{p-1}\alpha_{(l \cdot B)}(s_A(x_1,\ldots,x_k))&=
        \sum_{l=1}^{p-1}\alpha_{(l \cdot B)}\left(\sum_{\mu \in T}c_\mu x^\mu\right) \\
        &= \sum_{l=1}^{p-1}\left( \sum_{\mu \in T} c_\mu\alpha_{(l \cdot B)} (x^\mu)  \right)\\
        &= \sum_{\mu \in T}c_\mu\left( \sum_{l=1}^{p-1} \alpha_{(l \cdot B)} (x^\mu)  \right).
    \end{align*}
We distinguish two cases, first $p | \langle \mu , B \rangle$ and second $p \nmid \langle \mu, B \rangle$. In the first case, we have
\begin{align*}
    \alpha_{B}(x^\mu) &= \prod_{j=1}^k X^{\mu_j b_j} 
    = X^{\sum_{j=1}^k \mu_j b_j} 
    = X^{\langle \mu , B \rangle} 
    = 1,
\end{align*}
since $X^p=1$. This yields 
\[\sum_{l=1}^{p-1}\alpha_{l\cdot B}(x^\mu)=\sum_{l=1}^{p-1}1=p-1.\]
In the second case we have $\langle \mu, B \rangle = z\neq   0 \mod p$. For any $l \in (\F_p)^\times$ we observe
\begin{align*}
    \alpha_{(l \cdot B)}&=\prod_{j=1}^kX^{\mu_j \cdot b_j\cdot l} 
    = X^{l \langle \mu , B \rangle} 
    = X^{l \cdot z}
\end{align*}
Now, since multiplication with $z$ is an isomorphism on $\F_p^\times$, we have
\begin{align*}
    \sum_{l=1}^{p-1}X^{l \cdot z} &= \sum_{l=1}^{p-1}X^l 
    = -1,
\end{align*}
in $\Z[X]/\Phi_p(X)$.
There are, counted with multiplicity, $\alpha_B(s_A(x_1,\ldots,x_k))(1)=\frac{\prod_{i<j}(a_j-a_i)}{\prod_{i<j}(j-i)}$ many exponents in $s_A$. We finally conclude
\begin{align*}
    \sum_{\mu \in T}c_\mu\left( \sum_{l=1}^{p-1} \alpha_{(l \cdot B)} (x^\mu)  \right) &=
    \sum_{\mu \in T \colon p \mid \langle \mu, B \rangle} c_\mu\left( \sum_{l=1}^{p-1} \alpha_{(l \cdot B)} (x^\mu)  \right) + 
    \sum_{\mu \in T \colon p \nmid \langle \mu, B \rangle}c_\mu \left( \sum_{l=1}^{p-1} \alpha_{(l \cdot B)} (x^\mu)  \right) \\
    &= m_{A,B} \cdot (p-1) - \frac{\prod_{i<j}(a_j-a_i)}{\prod_{i<j}(j-i)} + m_{A,B} \\
    &= m_{A,B} \cdot p - \frac{\prod_{i<j}(a_j-a_i)}{\prod_{i<j}(j-i)}.
\end{align*}
\end{proof}

\begin{remark}\label{remarkProof}
    Note that the proof also works for every other list $b_1,\ldots,b_k$ of $B$. So in principle, one can choose the list of $B$ with the best properties. 

  Note further that \cref{EvaluationSumOfSchur} yields yet another proof for \cref{Chebotarev}, without an explicit evaluation at $X=1$. Since
\[
\frac{\prod_{i<j}(a_j-a_i)}{\prod_{i<j}(j-i)}
\]
is not divisible by $p$ in $\Z$, we have
\[
    \sum_{l=1}^{p-1}\alpha_{(l \cdot B)}(s_A(x_1,\ldots,x_k)) = m \cdot p - \frac{\prod_{i<j}(a_j-a_i)}{\prod_{i<j}(j-i)} \neq 0
\]
and thus $\alpha_B(s_A) \neq 0 \mod \Phi_p(X)$.

\end{remark}

Since the quantity $m_{A,B} \cdot p - \frac{\prod_{i<j}(a_j-a_i)}{\prod_{i<j}(j-i)} \in \Z[X]/\Phi_p(X)$ is an integer, it is easier to understand the rightmost map in \cref{fig:CyclotomicFinite}.

\begin{lemma}\label{thm:CommutingDiagrams}
    Let $p,q \in \Z$ be two distinct primes with $\ord_p(q)=p-1$. The diagram in \cref{fig:CyclotomicFinite} commutes.
\end{lemma}

\begin{proof}
    Both maps are the canonical maps with kernel $\langle \{q , \Phi_p(X)\}\rangle \subseteq \Z[X]$.

\end{proof}

\begin{theorem}\label{FourierFiniteFixedSets}
    Let $p,q \in \Z$ be two distinct primes with $\ord_p(q)=p-1$.
    Let $F_p \in \F_{q^r}^{p \times p}$ be the Fourier matrix in characteristic $q$, see \cref{def:FourierFiniteField}, and let $A,B \in \binom{\{0,1,\ldots,p-1\}}{k}$ be two index sets. If
    \[
    q \nmid \left( m_{A,B} \cdot p - \frac{\prod_{i<j}(a_j-a_i)}{\prod_{i<j}(j-i)}\right),
    \]
    we have $\det((F_p)_{A,B})\neq 0 \in \F_{q^r}$.
\end{theorem}

\begin{proof}
    Assume that $\det((F_p)_{A,B})=0 \in \F_{q^r}$, by \cref{shiftIndicesCharq} we know that
    $\det((F_p)_{A,l \cdot B})=0$ for all $l \in (\F_p)^\times$. Furthermore let $R_l(X) \in \Z[X]/\Phi_p(X)$ be the remainder of $\alpha_{l\cdot B}(s_A(x_1,\ldots,x_k))$ modulo $\Phi_p(X)$.
    We have $\det((F_p)_{A,l \cdot B})=0$ if and only if $q$ divides every coefficient of $R_l$ for every $l$. Hence we also know that $q$ divides the coefficients of the sum
    \[
    \sum_{l=1}^{p-1} R_l(X) = \left( m \cdot p - \frac{\prod_{i<j}(a_j-a_i)}{\prod_{i<j}(j-i)}\right),
    \]
    where the equality follows from \cref{EvaluationSumOfSchur}. This is a contradiction to our assumption and hence $\det((F_p)_{A,B}) \neq 0$.
    
\end{proof}

Our main result now follows directly.

\begin{Corollary}\label{FourierfiniteGreater}
    Let $p,q \in \Z$ be two distinct primes with $\ord_p(q)=p-1$.
    Let $F_p \in \F_{q^r}^{p \times p}$ be the Fourier matrix in characteristic $q$, see \cref{def:FourierFiniteField}.
    If
    \[
    q > \max_{A,B \subseteq \{0,1,\ldots,p-1\}, \#A=\#B} \left| m_{A,B} \cdot p - \frac{\prod_{i<j}(a_j-a_i)}{\prod_{i<j}(j-i)}\right|,
    \]
    then all minors of $F_p$ are nonzero.    
\end{Corollary}

\begin{proof}
    By \cref{FourierFiniteFixedSets} we need to show that
    \[
    q \nmid  \left( m_{A,B} \cdot p - \frac{\prod_{i<j}(a_j-a_i)}{\prod_{i<j}(j-i)}\right),
    \]
    for all $A,B \subseteq \{0,1,\ldots,p-1\}$, with $\#A=\#B$. This is trivially the case under the condition on $q$ and the statement follows.
\end{proof}


\begin{example}
    For $p=7$ we have
\[
\max_{A,B \subseteq \{0,1,\ldots,p-1\}, \#A=\#B} \left| m_{A,B} \cdot p - \frac{\prod_{i<j}(a_j-a_i)}{\prod_{i<j}(j-i)}\right|=8
\]
and thus \cref{FourierfiniteGreater} yields that the Fourier matrix $F_7$ has nonvanishing minors in characteristic $q$ for every prime $q>8$ with $\ord_7(q)=6$. In comparison, for $p=7$ \cite[Theorem A]{Zhang} yields this property of $F_7$ in characteristic $q$ for any prime $q\ge 75$ and $\ord_7(q)=6$.

\begin{table}[h!]
    \centering
    \begin{tabular}{l|c c c c c c}
         $p$ & 2 & 3 & 5 & 7 & 11 & 13 \smallskip \\
         \hline
         $q$ via \cref{FourierfiniteGreater} & 3 & 2 & 7 & 17 & 193&1619\\
         $q$ via \cite{PfCaLeMa25} & 3 & 2 & 11 & 79 & 105\,863 & 11\,709\,007\\
         $q$ via \cite{Zhang} & 3 & 2 & 13 & 89 & 105\,871 & 11\,709\,007\\
    \end{tabular}
    \caption{First prime $q$ for which $F_p$ has nonvanishing minors in characteristic $q$ according to \cref{FourierfiniteGreater}, the recent improvement \cite{PfCaLeMa25} of \cite{Zhang}, and \cite[Theorem A]{Zhang}.}
    \label{tab:my_label}
\end{table}
\end{example}

Removing the order condition has its own charme, in particular, if $q \equiv 1 \mod p$, the Fourier matrix $F_p$ has entries in the field $\F_q$. In this case, we can use the results from \cite{PfCaLeMa25}, where the order condition has been removed.
\begin{example}\label{ex:F5}
 For $p=5$ and $q=11$, we have $\ord_5(11)=1$ and thus already $\F_{11}$ contains 5-th roots of unity, e.g. $\omega_5=3$.
 According to \cite{PfCaLeMa25}, the Fourier matrix 
 \begin{align*}
 F_5=
 \left(
\begin{array}{ccccc}
1 & 1 & 1 & 1 & 1 \\
1 & 3 & 9 & 5 & 4 \\
1 & 9 & 4 & 3 & 5 \\
1 & 5 & 3 & 4 & 9 \\
1 & 4 & 5 & 9 & 3 \\
\end{array}
\right)\in\F_{11}^{5\times 5}
 \end{align*}
 has no vanishing minors in characteristic $11$. This also implies that every square submatrix is invertible over $\R$ and via standard arguments \cite{Tao} the uncertainty principle
   $  \|g\|_0 + \|F_5 g\|_0\ge 6$
 for all $g\in\R^5$, where $\|\cdot\|_0$ denotes the number of nonzero entries of a vector.
\end{example}

\subsection{Characteristic $q$ not primitive modulo $p$}

If we consider two distinct primes $p,q$ with $\ord_p(q)=r<p-1$, we cannot prove explicit bounds as in \cref{FourierfiniteGreater}. We start with a similar result as \cref{EvaluationSumOfSchur}, but for the nonprimitive case. This requires some technical definition.




\begin{definition}
    Let $p,q$ be two distinct primes with $r=\ord_p(q)$. Let $\F_q[X]/\overline{P(X)}$ be the $p$-th cyclotomic field over $F_q$. For $i=1,\ldots,p-1$ we know that the element $X^i$ is a root of unity and by $L_{(X^{i})}$ we denote the trace of the element $X^i$, i.e., the coefficient of the second highest monomial in the minimal polynomial of $X^i$ over $\F_q[X]$.
    
    Since $\langle q \rangle \subseteq (\F_p)^\times$ is a proper subgroup, we can look at the quotient group $(\F_p)^\times / \langle q \rangle$. This group consists of $(p-1)/r$ many cosets, where we denote a set of representatives by $n_1,\ldots,n_{(p-1)/r}$.
    
    Furthermore, by $m_{A,B}^{(i)}$, we denote the number of exponents $\mu$ of $s_A(x_1,\ldots,x_k)=\sum_{\mu \in T}x^\mu$, such that $i \equiv \langle \mu, B \rangle \in (\F_p)^\times / \langle q \rangle$.
\end{definition}
\begin{example}
    Let $p=7,q=2$, with $\ord_p(q)=3$. In $\F_2[X]$ we have
    \[
    \overline{\Phi_7(X)}=(X^3+X+1) \cdot (X^3+X^2+1)
    \]
    For $A=\{0,1,3\}$ it is $s_A(x_1,x_2,x_3)=x_1+x_2+x_3$ and the $\mu$ are given by $(1,0,0)$, $(0,1,0)$ and $(0,0,1)$. The cosets of the subgroup generated by $q=2$ in $\F_7^\times$ are $\{1,2,4\}$ and $\{3,5,6\}$. If we construct $F_{2^3}$ as $\F_2[X]/(X^3+X+1)$, we have $L_{X}=L_{X^2}=L_{X^4}=0$ and $L_{X^3}=L_{X^5}=L_{X^6}=1$. Lastly, for $B=\{0,2,4\}$ we have $m^{(0)}_{A,B}=1$, $m^{(1)}_{A,B}=2$ and $m^{(3)}_{A,B}=0$.
\end{example}

\begin{lemma}\label{lemmanonprim}
    Let $p,q\in \Z$ be two distinct primes, $r=\ord_p(q)$ and $A,B \in \binom{\{0,1,\ldots,p-1\}}{k}$ be two sets.
    We have
    \[
    \sum_{l=0}^{r-1}\alpha_{(q^l \cdot B)}(s_A(x_1,\ldots,x_k)) =r \cdot m^{(0)}_{A,B} - \sum_{i=1}^{(p-1)/r}m^{(n_i)}_{A,B} \cdot L_{(x^{n_i})} \in \F_q[X]/\overline{P(X)}.
    \]
\end{lemma}

\begin{proof}
    By the same argumentation as before, we obtain
    \[
    \sum_{l=0}^{r-1}\alpha_{(q^l\cdot B)}\left(\sum_{\mu \in T} c_\mu x^\mu \right)= \sum_{\mu \in T} c_\mu \sum_{l=0}^{r-1}\alpha_{(q^l\cdot B)}(x^\mu) \in \Z[X]
    \]
    The number $m^{(0)}_{A,B}$ denotes the number of indices $\mu$ with $p | \langle \mu,B\rangle$. In this case we have 
    $\alpha_{(q^l\cdot B)}(x^\mu) = 1$
    in $\Z[X]/\Phi_p(X)$. And thus
    \[
    \sum_{l=0}^{r-1}\alpha_{(q^l\cdot B)}(x^\mu)=r \in \F_q[X]/\overline{P(X)}.
    \]
    Now let $i \equiv \langle \mu, B\rangle \in (\F_p)^\times$, we obtain
    \begin{align}
        \alpha_{(q^l\cdot B)}(x^\mu) &= X^i \in \Z[X]/\Phi_p(X),
    \end{align}
    and furthermore
    \begin{align*}
        \sum_{l=0}^{r-1}\alpha_{(q^l\cdot B)}(x^\mu)&=\sum_{l=0}^{r-1}(X^{i})^{q^l}
    \end{align*}
    The elements $(X^{i})^{q^l} \in \F_{q^r}$ are conjugate, and there is no conjugate element missing in this list, and thus the sum yields the negative trace of this element
    \begin{align*}
        \sum_{l=0}^{r-1}(X^{i})^{q^l} &= - L_{(X^i)},
    \end{align*}
    which can be deduced from the fact that the polynomial
    \[
    \prod_{l=0}^{r-1}(Y-(X^i)^{q^l})
    \]
    is the minimal polynomial of all of these elements over $\F_q$. Lastly, since the sum iterates through all conjugate roots, the result is the same for each $j \in \F_p^\times$ that is conjugate to $i$ modulo the subgroup $\langle q \rangle$.
\end{proof}

\begin{figure}[b]
    \centering
        \[\begin{tikzcd}[row sep=large, column sep=large]
            \Z[X] \arrow{r}{ \mod \Phi_p(X)} \arrow{d}{\mod q}  &    \Z[X]/\Phi_p(X) \arrow{d}{\mod \langle q ,P(X)\rangle} \\
             \F_q[X] \arrow{r}{\mod \overline{P(X)}}  & \F_q[X]/\overline{P(X)}
            \end{tikzcd}
        \]
    \caption{The $p$-th cyclotomic polynomial over $\Z$ and $\F_q$, with $\ord_p(q)=r<p-1$.}
    \label{fig:CyclotomicFiniteSmallOrder}
\end{figure}

For $\ord_p(q)=r<p-1$, the diagram in \cref{fig:CyclotomicFiniteSmallOrder} looks a bit different than \cref{fig:CyclotomicFinite}. There still is a map $\varphi \colon \Z[X]/(\Phi_p(X)) \to \F_q[X]/\overline{P(X)}$ but with a larger kernel $\langle \{q,P(X)\}\rangle$.

The main problem with this approach under the condition $\ord_p(q)<p-1$ is that
the factorization of $\Phi_p(X)$ is dependent on $q$. Especially the traces of the elements $X^i \in \F_{q^r}$ are not known, while for $\ord_p(q)=p-1$ we know that the trace is $1$, independent of $q$. Thus any bound arising from \cref{lemmanonprim} will be still dependent on $q$.


\section*{Acknowledgement}
We sincerely thank Romanos Malikiosis for poining out a mistake in the previous version of Lemma 12 which had resulted in a flawed version of Corollary 17 for $\ord_p(q)<p-1$. The corrected version now is stated for the primitive case $\ord_p(q)=p-1$ only and Section 4.1 partially extends our approach to the nonprimitive case $\ord_p(q)<p-1$.


\bibliographystyle{abbrvurl}
\bibliography{refs}

\end{document}